\documentclass{article}
\usepackage[utf8]{inputenc}
\usepackage{graphicx}
\usepackage{amsmath}
\usepackage{amssymb}
\usepackage{enumerate}
\usepackage{latexsym}
\usepackage{epsfig}
\usepackage{amsthm}
\usepackage{authblk}

\newtheorem{definition}{Definition}
\newtheorem{lemma}{Lemma}
\newtheorem{theorem}{Theorem}

\newtheorem{corollary}{Corollary}

\title{
Generalized Fibonacci sequences and 
their properties}

\author[1]{Martin Bunder}
\author[2]{Joseph Tonien}
\affil[1]{School of Mathematics and Applied Statistics,
University of Wollongong,
Australia}
\affil[2]{School of Computing and Information Technology,
University of Wollongong,
Australia}

\begin{document}

\maketitle

\begin{abstract}
Let $F_n(k)$ be the generalized Fibonacci number defined by (with $F_i(k)$ abbreviated to $F_i$):
$F_n = F_{n-1} + F_{n-2} + \dots + F_{n-k}$, for $n \geq k$, and the initial values $(F_0,F_1,...,F_{k-1})$. Let $B_n(k,j)$ be $F_n(k)$ with initial values given by $F_j = 1$ and, for $i<j$ and $j<i<k$, $F_i = 0$. This paper shows that any $F_n(k)$ can be expressed as the sum of $B_n(k,j)$s. This paper also  expresses
$B_n(k,j)$ and $F_n(k)$ as finite sums, derives some properties and evaluates their 2-adic order
for a range of values of $k, j$ and $n$ and those of $B_n(3,j)$ and $B_n(4,j)$  for most values of $j$ and $n$.
\end{abstract}

\section{Introduction}

In this paper, for a fixed integer $k \geq 2$, we consider the generalized Fibonacci sequences $\{F_n(k)\}_{n \geq 0}$ given by the $k$-order Fibonacci recurrence equation
with $F_i(k)$ abbreviated to $F_i$:
\begin{equation}
    F_n = F_{n-1} + F_{n-2} + \dots + F_{n-k}, \mbox{ for all } n \geq k.
\end{equation}
Each of these $k$-order Fibonacci sequences is completely determined by the values of the first $k$ terms $(F_0, F_1, \dots, F_{k-1})$ of the sequence. The set of generalized Fibonacci sequences forms a $k$-dimensional vector space over $\mathbb{C}$.

 Lengyel and Marques in \cite{Lengyel_Marques_2014} and \cite{Lengyel_Marques}, Sobolewski in \cite{Sobolewski}, Young in \cite{Young} and Bunder and Tonien in \cite{Bunder2020} have considered the generalized Fibonacci sequence $\{T_n(k)\}_{n\geq 0}$ determined by the first $k$ terms $(0, 1, 1, \dots, 1)$ and evaluated the 2-adic order of its elements for many values of $k$.

Lengyel and Marques in~\cite{Lengyel_Marques} also mentioned the generalized Fibonacci sequence, denoted by $\{B_n(k)\}_{n \geq 0}$, defined by the first $k$ terms $(0,0, \dots, 0, 1)$. In this paper, we will consider a more general version of this sequence, which we denote as $\{B_n(k,j)\}_{n \geq 0}$. This sequence is parameterized by $k$ and $j$, where $0 \leq j \leq k-1$. The sequence is determined by the first $k$ terms $(0, \dots, 0,1,0, \dots, 0)$ where the only non-zero element in the first $k$ terms is $B_j(k,j)=1$. These sequences $\{B_n(k,j)\}_{n \geq 0}$ form the standard basis 
for the $k$-dimensional vector space of all $k$-order Fibonacci sequences, and thus, we have the following theorem.

\begin{theorem}
\label{Bkj:theorem0}
Let $k \geq 2$.
For any generalized Fibonacci sequence $\{F_n(k)\}_{n \geq 0}$
of order $k$, we have
$$
F_n(k) = \sum_{j=0}^{k-1} {F_j(k) B_n(k,j)}.
$$
\end{theorem}

We also consider the generalized Fibonacci sequence,  denoted by $\{S_n(k)\}_{n \geq 0}$, defined by the first $k$ terms $(1,1, \dots, 1)$. We will show some relationship between this sequence $\{S_n(k)\}_{n \geq 0}$ and the sequences $\{B_n(k,j)\}_{n \geq 0}$. We also determine the 2-adic order of $B_n(k,j)$ for many values of $k, j$ and $n$.

In our calculation of $v_2(B_n(k,j))$,
we will make use of the
binary digit sum function $s_2(n)$ which is defined as follows.

\begin{definition}
\label{Bkj:def1}
Define $s_2(0) = 0$.
If $i \geq 1$
and $n_1 > n_2 > \dots > n_i \geq 0$ then
$$
s_2(2^{n_1} + 2^{n_2} + \dots + 2^{n_i})=i.
$$
\end{definition}

We have the following lemma from~\cite{Bunder2020}.

\begin{lemma}
\label{Bkj:lemma1}

(i) $v_2(n!) = n - s_2(n)$.

(ii) $v_2({m \choose n})
= s_2(n) + s_2(m-n) - s_2(m)$.

(iii)
If $1 \leq n < 2^{v_2(m)}$ then
$s_2(m-n) = s_2(m) + v_2(m) - s_2(n) - v_2(n)$.

(iv)
For $n>0$,
$s_2(n-1) = s_2(n) + v_2(n) -1$.

(v)
For $n > m$,
$s_2(n) - s_2(m) = n-m - \sum_{i=m+1}^{n} v_2(i)$.
\end{lemma}

\section{Formulas for $S_n(k)$, $B_n(k,j)$ and $F_n(k)$}

We now derive a very useful relation for all  $k$-order Fibonacci sequences.

\begin{lemma}
\label{Bkj:lemma2}
Let $k \geq 2$.
For any generalized Fibonacci sequence $\{F_n(k)\}_{n \geq 0}$ of order $k$, we have 

(i) If $n \geq k+1$, $F_n = 2 F_{n-1} - F_{n-k-1}$

(ii) If $n \geq m \geq k+1$, 
$$F_n = 2^{n-m+1} F_{m-1} -\sum_{i=m-k-1}^{n-k-1}{2^{n-k-1-i} F_{i}}.$$
\end{lemma}
\begin{proof}
(i) If $n \geq k+1$, $F_n = F_{n-1} + \dots + F_{n-k}$ and $F_{n-1} = F_{n-2} + \dots + F_{n-k-1}$, so $F_n = 2 F_{n-1} - F_{n-k-1}$.

(ii) If $n \geq k+1$, by (i),
$$2^{-n} F_n - 2^{-(n-1)} F_{n-1}= - 2^{-n} F_{n-k-1}.$$
So for any $m \geq k+1$, taking the sum
$$\sum_{i=m}^{n}{(2^{-i} F_i - 2^{-(i-1)} F_{i-1})} = -\sum_{i=m}^{n}{2^{-i} F_{i-k-1}}$$
we obtain $$2^{-n} F_n - 2^{-(m-1)} F_{m-1} = -\sum_{i=m}^{n}{2^{-i} F_{i-k-1}}= -\sum_{i=m-k-1}^{n-k-1}{2^{-i-k-1} F_{i}}$$
and so $$F_n = 2^{n-m+1} F_{m-1} -\sum_{i=m-k-1}^{n-k-1}{2^{n-k-1-i} F_{i}}.$$
\end{proof}

The following theorem gives an explicit formula for $S_n(k)$.

\begin{theorem}
\label{Bkj:theorem1}
Let $k \geq 2$.
The $k$-order Fibonacci sequence $\{S_n(k)\}_{n \geq 0}$ determined by the first $k$ terms $(1,1,\dots, 1)$ satisfies the following formula
\begin{equation}
S_n(k) = 1 - (k-1) \sum_{i=1}^{\lfloor \frac{n+1}{k+1} \rfloor}{ (-1)^i {{n-ik} \choose {i-1}} 2^{n+1-i(k+1)}}.
\end{equation}
\end{theorem}
\begin{proof}
The desired formula can be proved by induction on $n$ based on Lemma~\ref{Bkj:lemma2}(i). 
The formula is correct for $0 \leq n\leq k$.
For $n \geq k+1$,
two separate cases need to be considered, $n+1 \equiv 0 \pmod{k+1}$ and $n+1 \not\equiv 0 \pmod{k+1}$.

Case 1: $n+1 \equiv 0 \pmod{k+1}$.
Write $n+1=a(k+1)$.

By Lemma~\ref{Bkj:lemma2}(i) and the induction hypothesis,
\begin{align*}
    S_n = &  2 S_{n-1} - S_{n-k-1}\\
    = & 2 - 2(k-1) \sum_{i=1}^{a-1}{ (-1)^i {{a-2+(a-i)k} \choose {i-1}} 2^{(a-i)(k+1)-1}}
    \\
    & -1 + (k-1) \sum_{i=1}^{a-1}{ (-1)^i {{a-2+(a-1-i)k} \choose {i-1}} 2^{(a-1-i)(k+1)}}
    \end{align*}
  
In the second summation, let $i:=i+1$, we have
  \begin{align*}  
    S_n
    = & 1 - (k-1) \sum_{i=1}^{a-1}{ (-1)^i {{a-2+(a-i)k} \choose {i-1}} 2^{(a-i)(k+1)}}
    \\
    &  - (k-1) \sum_{i=2}^{a}{ (-1)^i {{a-2+(a-i)k} \choose {i-2}} 2^{(a-i)(k+1)}}
    \\
    =&
     1 - (k-1) \sum_{i=1}^{a}{ (-1)^i {{a-1+(a-i)k} \choose {i-1}} 2^{(a-i)(k+1)}}
     \\
     =&
     1 - (k-1) \sum_{i=1}^{\lfloor \frac{n+1}{k+1} \rfloor}{ (-1)^i {{n-ik} \choose {i-1}} 2^{n+1-i(k+1)}}.
\end{align*}

Case 2: $n+1 \not\equiv 0 \pmod{k+1}$.
Write $n+1=a(k+1) + r$ where $1 \leq r \leq k$.

By Lemma~\ref{Bkj:lemma2}(i) and the induction hypothesis,
\begin{align*}
    S_n = &  2 S_{n-1} - S_{n-k-1}\\
    =&
    2 -  2(k-1) \sum_{i=1}^{a}{ (-1)^i {{a + r-2+(a-i)k} \choose {i-1}} 2^{ r+(a-i)(k+1)-1}}
    \\
    &-1 + (k-1) \sum_{i=1}^{a-1}{ (-1)^i {{a + r-2+(a-1-i)k} \choose {i-1}} 2^{ r+(a-1-i)(k+1)}}
        \end{align*}
  
In the second summation, let $i:=i+1$, we have
\begin{align*}
S_n
    =&
    1 -  (k-1) \sum_{i=1}^{a}{ (-1)^i {{a + r-2+(a-i)k} \choose {i-1}} 2^{ r+(a-i)(k+1)}}
    \\
    & - (k-1) \sum_{i=2}^{a}{ (-1)^i {{a + r-2+(a-i)k} \choose {i-2}} 2^{ r+(a-i)(k+1)}}
    \\
     =&
     1 - (k-1) \sum_{i=1}^{a}{ (-1)^i {{a + r-1+(a-i)k} \choose {i-1}} 2^{ r+(a-i)(k+1)}}.
    \\
     =&
     1 - (k-1) \sum_{i=1}^{\lfloor \frac{n+1}{k+1} \rfloor}{ (-1)^i {{n-ik} \choose {i-1}} 2^{n+1-i(k+1)}}.
\end{align*}

\end{proof}

The sequences $\{B_n(k,j)\}_{n \geq 0}$ are related to 
$\{T_n(k)\}_{n \geq 0}$ and $\{S_n(k)\}_{n \geq 0}$ as follows:

\begin{lemma}
\label{Bkj:lemma3}
Let $k \geq 2$, $0 \leq j \leq k-1$.
We have

(i) $T_n(k) = \sum_{j=1}^{k-1}{B_n(k,j)}$ for all $n \geq 0$.

(ii) $S_n(k) = \sum_{j=0}^{k-1}{B_n(k,j)}$ for all $n \geq 0$.

(iii) $B_n(k,j) = \frac{1}{k-1} (S_n(k) - S_{n-j-1}(k))$ for all $n \geq j+1$.
\end{lemma}
\begin{proof}
The results (i) and (ii) are derived from the linearity of the $k$-order Fibonacci sequences.

(iii) Using the $k$-order Fibonacci recurrence equation, extend the sequence $S_n(k)$ for negative values of $n$. Then we have $S_{-1} = -(k-2)$, $S_{-2} = S_{-3} = \dots = S_{-k} = 1$. 
Define the $k$-order Fibonacci sequence $\{B'_n\}_{n \geq 0}$ as $B'_n = \frac{1}{k-1}(S_n - S_{n-j-1})$. Then all the first $k$ terms of $B'_n$ are zero except $B'_j=1$. It follows that $B_n$ and $B'_n$ are the same $k$-order Fibonacci sequence.
\end{proof}

Using Lemma~\ref{Bkj:lemma3}(iii) and the formula for $S_n$ in Theorem~\ref{Bkj:theorem1}, we derive the following explicit formula for $B_n(k,j)$.

\begin{theorem}
\label{Bkj:theorem2}
Let $k \geq 2$, $0 \leq j \leq k-1$.
The $k$-order Fibonacci sequence $\{B_n(k,j)\}_{n \geq 0}$ satisfies the following formula, for all $n \geq j+1$,
\begin{align*}
B_n(k,j) = & - \sum_{i=1}^{\lfloor \frac{n+1}{k+1} \rfloor}{ (-1)^i {{n-ik} \choose {i-1}} 2^{n+1-i(k+1)}}
\\
&
+ \sum_{i=1}^{\lfloor \frac{n-j}{k+1} \rfloor}{ (-1)^i {{n-j-1-ik} \choose {i-1}} 2^{n-j-i(k+1)}}.
\end{align*}
\end{theorem}

Using Theorem~\ref{Bkj:theorem2}, we can calculate 
all the values of $B_n(k,j)$. Below are some examples:

\begin{lemma}
\label{Bkj:lemma4}
Let $k \geq 2$, $0 \leq j \leq k-1$.
We have

(i) If $k \leq n \leq k+j$, $$B_n(k,j) = 2^{n-k}.$$

(ii) If $k+j+1 \leq n \leq 2k$, $$B_n(k,j) = 2^{n-k} - 2^{n-k-j-1}.$$

(iii) If $2k+1 \leq n \leq 2k+j+1$, $$B_n(k,j) = 2^{n-k} - 2^{n-k-j-1} - (n-2k) 2^{n-2k-1}.$$

(iv) If $2k+j+2 \leq n \leq 3k+1$, 
$$B_n(k,j)=
2^{n-k} - 2^{n-k-j-1} - (n-2k) 2^{n-2k-1} + (n-2k-j-1) 2^{n-2k-j-2}
.$$

(v)  If $3k+2 \leq n \leq 3k+j+2$,
\begin{align*}
B_n(k,j)= &
 2^{n-k} 
- 2^{n-k-j-1} 
- (n-2k) 2^{n-2k-1} 
\\
&
+ (n-2k-j-1) 2^{n-2k-j-2}  
  + \frac{(n-3k-1)(n-3k)}{2}  2^{n-3k-2} 
.
\end{align*}
\end{lemma}

Combining Theorem~\ref{Bkj:theorem0}
and Theorem~\ref{Bkj:theorem2}, we derive the following
explicit formula for
$F_n(k)$.

\begin{theorem}
\label{Bkj:theorem0B}
Let $k \geq 2$.
Any generalized Fibonacci sequence $\{F_n(k)\}_{n \geq 0}$
of order $k$
satisfies the following formula, for all $n \geq k$,
\begin{align*}
F_n(k) & = -  F_k(k)  
  \sum_{i=1}^{\lfloor \frac{n+1}{k+1} \rfloor}{ (-1)^i {{n-ik} \choose {i-1}} 2^{n+1-i(k+1)}}
\\
&
+ \sum_{j=0}^{k-1} \sum_{i=1}^{\lfloor \frac{n-j}{k+1} \rfloor}{ (-1)^i F_j(k) {{n-j-1-ik} \choose {i-1}} 2^{n-j-i(k+1)}}
\end{align*}
\end{theorem}

\section{The 2-adic order of $B_n(k,j)$}

Using the formula for $B_n(k,j)$
in Theorem~\ref{Bkj:theorem2}, we obtain the following theorem.

\begin{theorem}
\label{Bkj:theorem3}
Let $k \geq 2$, $0 \leq j \leq k-1$, $n = a(k+1) + r$, $-1 \leq r \leq k-1$.

(i) If $-1 \leq r \leq j-1$ and $a \geq 2$,
\begin{align*}
B_n(k,j)
= &
2^{r+2k+2-j} C_1
+
2^{r+k+2} C_2
-  (-1)^{a} {{a + r+k-j-1} \choose {a-2}} 2^{r+k+1-j}
\\
&
 -  (-1)^{a} {{a + r} \choose {a-1}} 2^{r+1},
\end{align*}
where $C_1$ and $C_2$ are integers.

(ii) If $j \leq r \leq k-1$,
\begin{align*}
B_n(k,j) 
= & 
2^{r+k+2} C_1
+ 2^{ r+k+1-j} C_2
- (-1)^a {{a + r} \choose {a-1}} 2^{r+1}
\\
&
+  (-1)^a {{a+ r-j-1} \choose {a-1}} 2^{ r-j},
\end{align*}
where $C_1$ and $C_2$ are integers.
\end{theorem}

Inspired by Theorem~\ref{Bkj:theorem3},
we have the following definition.

\begin{definition}
\label{Bkj:def2}
Let $k \geq 2$, $0 \leq j \leq k-1$, $n = a(k+1) + r$, $-1 \leq r \leq k-1$. 

(i) If $-1 \leq r \leq j-1$ and $a \geq 2$, define
\begin{align*}
\Delta_1 
= & 
(r+k+2) - 
v_2 \left(
{{a+r} \choose {a-1}}
2^{r+1}
\right),
\\
\Delta_2 = &
v_2 \left(
{{a + r+k-j-1} \choose {a-2}} 2^{r+k+1-j}
\right)
-
v_2 \left(
{{a+r} \choose {a-1}}
2^{r+1}
\right)
\end{align*}

(ii) If $j \leq r \leq k-1$, define
\begin{align*}
\Delta_3
= &
(r+k+1-j)-
v_2 \left( 
{{a+ r-j-1} \choose {a-1}}
2^{r-j}
\right)
,
\\
\Delta_4
=&
v_2 \left( 
{{a+ r} \choose {a-1}}
2^{r+1}
\right)
-
v_2 \left( 
{{a+ r-j-1} \choose {a-1}}
2^{r-j}
\right)
.
\end{align*}
\end{definition}

The following theorem is a direct consequence of Theorem~\ref{Bkj:theorem3}.
It allows us to determine $v_2(B_n(k,j))$
for a large range of values of $k$, $j$ and $n$.

\begin{theorem}
\label{Bkj:theorem4}
Let $k \geq 2$, $0 \leq j \leq k-1$, $n = a(k+1) + r$, $-1 \leq r \leq k-1$.

(i) If $r=-1$ or $r=j$ then $
v_2(B_n(k,j)) = 0$.

(ii) If $0 \leq r \leq j-1$ and $a \geq 2$, $\Delta_1 > 0$ and $\Delta_2 > 0$ then
$$
v_2(B_n(k,j))
=
r+1+
v_2 \left( 
{{a + r} \choose {a-1}}
\right).
$$

(iii) If $j+1 \leq r \leq k-1$, $\Delta_3 >0$ and $\Delta_4 >0$
then
$$
v_2(B_n(k,j))
=
r-j+
v_2 \left( 
{{a+ r-j-1} \choose {a-1}}
\right).
$$
\end{theorem}

Applying Theorem~\ref{Bkj:theorem4} for some particular values of $r$ we obtain the following theorem.

\begin{theorem}
\label{Bkj:theorem5}
Let $k \geq 2$, $0 \leq j \leq k-1$, $n = a(k+1) + r$, $-1 \leq r \leq k-1$.

(i) If $r=0$, $j \geq 1$, 
$a$ is odd then
$
v_2(B_n(k,j))
=
1
$.

(ii) If $r=0$, $1 \leq j \leq k-2$, 
$a$ is even and $v_2(a) \leq k$ then
$
v_2(B_n(k,j))
=
1+
v_2(a)
$.

(iii) If $r=1$, $j \geq 2$,
$a$ is odd
and $v_2(a+1) \leq k+1$
then
$
v_2(B_n(k,j))
=
1+
v_2(a+1)
$.

(iv) If $r=1$, $j \geq 2$
$a$ is even
and $v_2(a) \leq k+1$
then
$
v_2(B_n(k,j))
=
1+
v_2(a)
$.

(v) If $r=1$, $j=0$,
$a$ is odd
then
$
v_2(B_n(k,j))
=
1
$.

(vi) If $r=j+1 \leq k-1$, $j\geq 1$
and $v_2(a) \leq k$ then
$
v_2(B_n(k,j))
=
1 + v_2(a)
$.
\end{theorem}
\begin{proof}
(i)
When $a=1$, it follows from Lemma~\ref{Bkj:lemma4}(i).

When $a$ is odd and $a > 1$, apply
Theorem~\ref{Bkj:theorem4}(ii) with $r=0$.
We have
$$\Delta_1=
(k+2) - 
v_2 \left(
{{a} \choose {a-1}}
2^{1}
\right)
=
k+1  > 0$$
and
\begin{align*}
\Delta_2 
& =
v_2 \left(
{{a +k-j-1} \choose {a-2}} 2^{k+1-j}
\right)
-
v_2 \left(
{{a} \choose {a-1}}
2^{1}
\right)
\\
& =  k-j
+ v_2 \left( {{a+k-j-1} \choose {a-2}} \right) > 0.
\end{align*}
Therefore,
$
v_2(B_n(k,j))
=
1
$.
   
(ii)
Apply
Theorem~\ref{Bkj:theorem4}(ii) with $r=0$.
As $a$ is even,
we have
$$\Delta_1
=
(k+2) - 
v_2 \left(
{{a} \choose {a-1}}
2^{1}
\right)
=
k+1-v_2(a)  > 0$$
and
\begin{align*}
\Delta_2 
& =
v_2 \left(
{{a +k-j-1} \choose {a-2}} 2^{k+1-j}
\right)
-
v_2 \left(
{{a} \choose {a-1}}
2^{1}
\right)
\\
& =  k-j
-v_2(k-j)-v_2(k-j+1)
+ v_2 \left( {{a+k-j-1} \choose {a}} \right)
\\
& \geq 
k-j
-v_2(k-j)-v_2(k-j+1) > 0.
\end{align*}
Therefore, 
$
v_2(B_n(k,j))
=
1+
v_2(a)
$.
  
(iii)
When $a=1$, it follows from Lemma~\ref{Bkj:lemma4}(i).

When $a$ is odd and $a > 1$, apply
Theorem~\ref{Bkj:theorem4}(ii) with $r=1$.
We have
$$\Delta_1 
=
(k+3) - 
v_2 \left(
{{a+1} \choose {a-1}}
2^{2}
\right)
= k+2 -v_2(a+1) > 0$$
and 
\begin{align*}
\Delta_2 
=&
v_2 \left(
{{a +k-j} \choose {a-2}} 2^{k+2-j}
\right)
-
v_2 \left(
{{a+1} \choose {a-1}}
2^{2}
\right)
\\
= & k-j+1 + v_2(a-1)
+v_2 \left( {{a+k-j} \choose {a+1}} \right)
\\
& - v_2(k-j) -v_2(k-j+1) - v_2(k-j+2)
\\
 \geq &
k-j+2
- v_2(k-j) -v_2(k-j+1) - v_2(k-j+2) > 0.
\end{align*}
Therefore,
$v_2(B_n(k,j)) 
= 1 + v_2(a+1)$.

(iv)
Apply Theorem~\ref{Bkj:theorem4}(ii) with $r=1$.
As $a$ is even,
we have
$$\Delta_1 
=
(k+3) - 
v_2 \left(
{{a+1} \choose {a-1}}
2^{2}
\right)
= k+2 -v_2(a) > 0$$
and 
\begin{align*}
\Delta_2 
=&
v_2 \left(
{{a +k-j} \choose {a-2}} 2^{k+2-j}
\right)
-
v_2 \left(
{{a+1} \choose {a-1}}
2^{2}
\right)
\\
= & k-j+1 
+v_2 \left( {{a+k-j} \choose {a}} \right)
-v_2(k-j+1)-v_2(k-j+2)
\\
\geq  & k-j+1 
-v_2(k-j+1)-v_2(k-j+2)
> 0.
\end{align*}
Therefore,
$v_2(B_n(k,j)) 
= 1 + v_2(a)$.

(v)
When $a=1$, it follows from Lemma~\ref{Bkj:lemma4}(i).

When $a$ is odd and $a >1$,
apply Theorem~\ref{Bkj:theorem4}(iii) with $j=0$ and $r=1$.
We have
$$\Delta_3 
=(k+2)-
v_2 \left( 
{{a} \choose {a-1}}
2^{1}
\right)
= k+1>0$$
and
$$\Delta_4
=
v_2 \left( 
{{a+ 1} \choose {a-1}}
2^{2}
\right)
-
v_2 \left( 
{{a} \choose {a-1}}
2^{1}
\right)
= v_2(a+1) > 0.$$
Therefore,
$v_2(B_n(k,j)) = 1$.

(vi)
Apply
Theorem~\ref{Bkj:theorem4}(iii) with $r=j+1$.
We have
$$\Delta_3
=
(k+2)-
v_2 \left( 
{{a} \choose {a-1}}
2^{1}
\right)
=
k+1 -v_2(a) > 0$$
and
\begin{align*}
\Delta_4 
& =
v_2 \left( 
{{a+ j+1} \choose {a-1}}
2^{j+2}
\right)
-
v_2 \left( 
{{a} \choose {a-1}}
2^{1}
\right)
\\
& =
v_2 \left( {{a+j+1} \choose {a}} \right) 
+ j + 1 - v_2(j+2)
\\
& \geq j + 1 - v_2(j+2).
\end{align*}
As $j \geq 1$, $v_2(j+2) \leq j$
so $\Delta_4 > 0$ and
therefore,
$
v_2(B_n(k,j))
=
1 + v_2(a)
$.
\end{proof}

The following Lemma~\ref{Bkj:lemma5}
and Lemma~\ref{Bkj:lemma6}
are supplementary to
Theorem~\ref{Bkj:theorem5}(ii).

\begin{lemma}
\label{Bkj:lemma5}
Let $1 \leq j < k-1$, $n = a(k+1)$ and $v_2(a) =k+1$ then
$$
v_2(B_n(k,j))
=
k+2
.$$
\end{lemma}
\begin{proof}
When $n=a(k+1)$ and $1 \leq j < k-1$,
the formula for $B_n(k,j)$
in Theorem~\ref{Bkj:theorem2}
becomes
\begin{align*}
B_n(k,j) = & - \sum_{i=1}^{a}{ (-1)^i {{a(k+1)-ik} \choose {i-1}} 2^{(a-i)(k+1)+1}}
\\
&
+ \sum_{i=1}^{a-1}{ (-1)^i {{a(k+1)-j-1-ik} \choose {i-1}} 2^{(a-i)(k+1)-j}}.
\end{align*}

As $a$ is even, we have 
\begin{align*}
B_n(k,j)
= &
-2a + {{k+a} \choose {a-2}} 2^{k+2}
+
2^{2k+3} C_1
\\
&
-  {{a +k-j-1} \choose {a-2}} 2^{k+1-j}
+
2^{2k+2-j} C_2
 ,
\end{align*}
where $C_1$ and $C_2$ are integers.
We will show that
\begin{equation}
\label{B:EQN1}
    v_2 \left( {{k+a} \choose {a-2}} 2^{k+2} \right) 
   >  v_2(2a) = k+2
\end{equation}
and
\begin{equation}
\label{B:EQN2}
    v_2 \left( 
{{a +k-j-1} \choose {a-2}} 2^{k+1-j}
\right)
> v_2(2a) = k+2,
\end{equation}
therefore, 
$
v_2(B_n(k,j))
= v_2(2a)=
k+2
$.

To prove (\ref{B:EQN1}),
use Lemma~\ref{Bkj:lemma1}(ii), we have
\begin{align*}
   v_2 \left( {{k+a} \choose {a-2}}  \right) 
   = &  s_2(a-2) + s_2(k+2) - s_2(k+a).
\end{align*}

Write $a = 2^{k+1} a'$ where $a'$ is an odd number.
Then $s_2(a-2) = s_2(a') + k-1$ and $s_2(k+a) = s_2(k) + s_2(a')$.
So $$v_2 \left( {{k+a} \choose {a-2}}  \right) 
   =  k -1 + s_2(k+2) - s_2(k) \geq  s_2(k+2) \geq 1,$$
 and we obtain (\ref{B:EQN1}).
 
 To prove (\ref{B:EQN2}),
write
\begin{align*}
v_2 \left( 
{{a +k-j-1} \choose {a-2}}
\right)
= &
v_2 \left( 
{{a +k-j-1} \choose {a}}
\right)
\\
&
+v_2(a)
-v_2(k-j)
-v_2(k-j+1)
\\
\geq &
v_2(a)
-v_2(k-j)
-v_2(k-j+1)
.
\end{align*}

As $k-j>1$,
$$v_2(k-j)
+v_2(k-j+1) < k-j,$$
so we have
$$
v_2 \left( 
{{a +k-j-1} \choose {a-2}} 
\right)
> v_2(a) - k + j,
$$
 and we obtain (\ref{B:EQN2}).
\end{proof}

\begin{lemma}
\label{Bkj:lemma6}
Let $k \geq 2$, $j= k-1$, $n = a(k+1)$,
$a$ is even and $v_2(a) \leq k/2$ then
$$
v_2(B_n(k,j))
=1 + 2 v_2(a)
.$$
\end{lemma}
\begin{proof}
When $n=a(k+1)$ and $j=k-1$,
the formula for $B_n(k,j)$
in Theorem~\ref{Bkj:theorem2}
becomes
\begin{align*}
B_n(k,j) = & - \sum_{i=1}^{a}{ (-1)^i {{a(k+1)-ik} \choose {i-1}} 2^{(a-i)(k+1)+1}}
\\
&
+ \sum_{i=1}^{a-1}{ (-1)^i {{a(k+1)-k-ik} \choose {i-1}} 2^{(a-1-i)(k+1)+2}}.
\end{align*}

As $a$ is even, we have
\begin{align*}
B_n(k,j)
= &
-2a + 2^{k+2} C_1
-2a(a-1) + 2^{k+3} C_2
\\
= &
-2 a^2 +  2^{k+2} C_1
 + 2^{k+3} C_2,
\end{align*}
where $C_1$ and $C_2$ are integers.

Since $v_2(2a^2) = 1 + 2 v_2(a) < k+2$,
$v_2(B_n(k,j)) = v_2(2a^2)=1 + 2 v_2(a)$.
\end{proof}

The following Lemma~\ref{Bkj:lemma7}
is
 supplementary to
Theorem~\ref{Bkj:theorem5}(v).

\begin{lemma}
\label{Bkj:lemma7}
Let $k \geq 2$, $j=0$, $n=a(k+1) + 1$,
$a$ is even and $v_2(a) \leq k/2$
then
$
v_2(B_n(k,j))
=
1 + 2 v_2(a)
$.
\end{lemma}
\begin{proof}
When $n=a(k+1) + 1$ and $j=0$,
the formula for $B_n(k,j)$
in Theorem~\ref{Bkj:theorem2}
becomes
\begin{align*}
B_n(k,j) = & - \sum_{i=1}^{a}{ (-1)^i {{a(k+1) + 1-ik} \choose {i-1}} 2^{(a-i)(k+1) + 2}}
\\
&
+ \sum_{i=1}^{a}{ (-1)^i {{a(k+1) -ik} \choose {i-1}} 2^{(a-i)(k+1) + 1}}.
\end{align*}

As $a$ is even, we have
\begin{align*}
B_n(k,j) 
= &
-2a(a+1) + 2^{k+3} C_1 + 2a + 2^{k+2} C_2
\\
= &
-2a^2 + 2^{k+3} C_1 + 2^{k+2} C_2
\end{align*}
where $C_1$ and $C_2$ are integers.

Since $v_2(2a^2) = 1 + 2 v_2(a) < k+2$,
$v_2(B_n(k,j)) = v_2(2a^2)=1 + 2 v_2(a)$.
\end{proof}

\section{The 2-adic order of $F_n(k)$}

Using the formula for $F_n(k)$
in Theorem~\ref{Bkj:theorem0B}, we obtain the following theorem.

\begin{theorem}
\label{Bkj:thmFn1}
Let $k \geq 2$, 
$n = a(k+1) + r$, $-1 \leq r \leq k-1$.

(i) If $r=-1$ and $a \geq 2$, for all $n \geq k$,
\begin{align*}
F_n(k) & =  
-
(-1)^a F_k(k)
-  
2^{k+1} F_k(k) C_1
- (-1)^{a} \sum_{j=0}^{k-1} F_j(k)
    {{k + a-j-2} \choose {a-2}} 2^{k-j}  
+ 2^{k+2} C_2,
\end{align*}
where $C_1$ and $C_2$ are integers.

(ii)
If $r \geq 0$ and $a \geq 2$,
for all $n \geq k$,
\begin{align*}
F_n(k) 
& = -  (-1)^a F_k(k)  
  {{a+r} \choose {a-1}} 2^{r+1}
+ (-1)^a \sum_{j=0}^{r} F_j(k)
    {{a+r-j-1} \choose {a-1}} 2^{r-j}
\\
&
- (-1)^{a} \sum_{j=r+1}^{k-1} F_j(k)
    {{k+a+r-j-1} \choose {a-2}} 2^{k+1+r-j}
+ 2^{k+1} C
\end{align*}
where $C$ is an integer.
\end{theorem} 
 \begin{proof}
 (i)
 If $r=-1$, $n = a(k+1) -1$, then using Theorem~\ref{Bkj:theorem0B},
 we have
 \begin{align*}
F_n(k) & = -  F_k(k)  
  \sum_{i=1}^{a}{ (-1)^i {{(a-i) k  +a -1} \choose {i-1}} 2^{(a-i)(k+1)}}
\\
&
+ \sum_{j=0}^{k-1} \sum_{i=1}^{a-1}{ (-1)^i F_j(k) {{(a-i)k +a -j-2} \choose {i-1}} 2^{(a-i)(k+1)-j-1}}
\\
& = -  F_k(k)  
\left(
(-1)^a + 2^{k+1} C_1
\right) 
\\
&
+ \sum_{j=0}^{k-1} F_j(k)
\left(
(-1)^{a-1}
{{k +a -j-2} \choose {a-2}} 2^{k-j}
+ 2^{2(k+1)-j-1} C
\right)
\\
& =  
-
(-1)^a F_k(k)
-  
2^{k+1} F_k(k) C_1
\\
&
- (-1)^{a} \sum_{j=0}^{k-1} F_j(k)
    {{k+ a-j-2} \choose {a-2}} 2^{k-j}  
+ 2^{k+2} C_2.
\end{align*}

 (ii)
 If $r \geq 0$,  then using Theorem~\ref{Bkj:theorem0B},
 we have  
 \begin{align*}
F_n(k) & = -  F_k(k)  
  \sum_{i=1}^{a}{ (-1)^i {{(a-i)k+a+r} \choose {i-1}} 2^{(a-i)(k+1)+r+1}}
\\
&
+ \sum_{j=0}^{k-1} F_j(k) \sum_{i=1}^{\lfloor \frac{a(k+1)+r-j}{k+1} \rfloor}{ (-1)^i   {{(a-i)k+a+r-j-1} \choose {i-1}} 2^{(a-i)(k+1)+r-j}}
\\
& = -  F_k(k)  
  \sum_{i=1}^{a}{ (-1)^i {{(a-i)k+a+r} \choose {i-1}} 2^{(a-i)(k+1)+r+1}}
\\
&
+ \sum_{j=0}^{r} F_j(k) \sum_{i=1}^{a}{ (-1)^i   {{(a-i)k+a+r-j-1} \choose {i-1}} 2^{(a-i)(k+1)+r-j}}
\\
&
+ \sum_{j=r+1}^{k-1} F_j(k) \sum_{i=1}^{a-1}{ (-1)^i   {{(a-i)k+a+r-j-1} \choose {i-1}} 2^{(a-i)(k+1)+r-j}}
\\
& = -  F_k(k)  
\left(
(-1)^a {{a+r} \choose {a-1}} 2^{r+1}
+
2^{k+r+2} C_1
\right) 
\\
&
+ \sum_{j=0}^{r} F_j(k)
\left(
(-1)^a   {{a+r-j-1} \choose {a-1}} 2^{r-j}
+ 2^{k+1+r-j} C_2
\right) 
\\
&
+ \sum_{j=r+1}^{k-1} F_j(k)
\left(
(-1)^{a-1}   {{k+a+r-j-1} \choose {a-2}} 2^{k+1+r-j}
+ 2^{2k+2+r-j} C_3
\right)
\\
& = -  (-1)^a F_k(k)  
  {{a+r} \choose {a-1}} 2^{r+1}
\\
&
+ (-1)^a \sum_{j=0}^{r} F_j(k)
    {{a+r-j-1} \choose {a-1}} 2^{r-j}
\\
&
- (-1)^{a} \sum_{j=r+1}^{k-1} F_j(k)
    {{k+a+r-j-1} \choose {a-2}} 2^{k+1+r-j}
+ 2^{k+1} C.
\end{align*}
 \end{proof}
 
\begin{theorem}
\label{Bkj:thmFn2}
Let $k \geq 2$, $n=a(k+1) + r$,
$-1 \leq r \leq k-1$.

(i) If $r=-1$ and $v_2(F_k(k))=0$  then $v_2(F_n(k))=0$.

(ii)
If $r \geq 0$
and $v_2\left( \sum_{j=0}^{r} F_j(k)
    {{a+r-j-1} \choose {a-1}} 2^{r-j} \right) \leq r$
then
$$v_2(F_n(k))=
v_2 \left( \sum_{j=0}^{r} F_j(k)
    {{a+r-j-1} \choose {a-1}} 2^{r-j}
\right).
$$
\end{theorem}
\begin{proof}
(i) follows from Theorem~\ref{Bkj:thmFn1}(i).

(ii)
If $r \geq 0$,
by Theorem~\ref{Bkj:thmFn1}(ii),
\begin{align*}
F_n(k)
& =  2^{r+1} C
+ (-1)^a \sum_{j=0}^{r} F_j(k)
    {{a+r-j-1} \choose {a-1}} 2^{r-j},
\end{align*}
and thus,
if $v_2\left( \sum_{j=0}^{r} F_j(k)
    {{a+r-j-1} \choose {a-1}} 2^{r-j} \right) \leq r$
then
$$v_2(F_n(k))=
v_2 \left( \sum_{j=0}^{r} F_j(k)
    {{a+r-j-1} \choose {a-1}} 2^{r-j}
\right).$$
\end{proof}

We have the following corollary of
Theorem~\ref{Bkj:thmFn2}(ii).

\begin{corollary}
Let $k \geq 2$, $n=a(k+1) + r$,
$-1 \leq r \leq k-1$.

(i) If $r \geq 0$ and $v_2(F_r(k))=0$  then $v_2(F_n(k))=0$.

(ii) If $r \geq 1$, $v_2(F_r(k))=1$ 
and $v_2(a F_{r-1}(k)) > 0$
then $v_2(F_n(k))=1$.

(iii) If $r \geq 2$, $v_2(F_r(k) + 2a F_{r-1}(k)) \leq 2$ 
and $v_2\left( {{a+1} \choose {a-1}} F_{r-2}(k) \right) > 0$
then $v_2(F_n(k))=v_2(F_r(k) + 2a F_{r-1}(k))$.

(iv) If $r \geq i$, $v_2\left(F_r(k) + 2a F_{r-1}(k) + 2^2 {{a+1} \choose {a-1}} F_{r-2}(k) + \dots + 2^{i-1} {{a+i-2} \choose {a-1}} F_{r-i+1}(k) \right) \leq i$ 
and $v_2\left( {{a+i-1} \choose {a-1}} F_{r-i}(k) \right) > 0$
then \\
$v_2(F_n(k))=
v_2\left(F_r(k) + 2a F_{r-1}(k) + 2^2 {{a+1} \choose {a-1}} F_{r-2}(k) + \dots + 2^{i-1} {{a+i-2} \choose {a-1}} F_{r-i+1}(k) \right)$.
\end{corollary}

\section{The $k=3$ case}

\subsection{k=3, j=0}

\begin{theorem}
\label{Bkj:theoremk3j0}
For any $n \geq 0$,
$$
v_2(B_n(3,0))
=
\begin{cases}
0, 
& \text{if } 
n \equiv 0 \pmod{4},
\\
1, 
& \text{if } 
n \equiv 5 \pmod{8},
\\
3, 
& \text{if } 
n \equiv 9 \pmod{16},
\\
1+
v_2 \big( a(a+1) \big), 
& \text{if } 
n = 4a  + 2
\text{ and }
v_2 \big( a(a+1) \big)
\leq 4,
\\
0, 
& \text{if } 
n \equiv 3 \pmod{4}.
  \end{cases}
$$
\end{theorem}
\begin{proof}
When $n \equiv 0 \pmod{4}$,  it
follows from
Theorem~\ref{Bkj:theorem4}(i)
with $k=3$ and $j=0$
that
$
v_2(B_n(3,0)) = 0$.

When $n \equiv 5 \pmod{8}$, it
follows from
Theorem~\ref{Bkj:theorem5}(v) with $k=3$ and $j=0$
that $
v_2(B_n(3,0))
=
1
$.

When $n \equiv 9 \pmod{16}$,  it
follows from
Lemma~\ref{Bkj:lemma7}
with $k=3$ and $j=0$
that
$
v_2(B_n(3,0))
=
3
$.

When $n = 4a  + 2$,
we apply Theorem~\ref{Bkj:theorem4}(iii)
with $k=3$, $j=0$ and $r=2$.
In this case, 
$\Delta_3
= 
6-
v_2 \left( 
{{a+ 1} \choose {a-1}}
2^{2}
\right)
= 
5-
v_2 \big( a(a+1) \big)
$
and
$
\Delta_4
=
1+
v_2(a+2)
$.
So if 
$v_2 \big( a(a+1) \big)
\leq 4$,
then
$\Delta_3 >0$ and $\Delta_4 >0$,
and thus by Theorem~\ref{Bkj:theorem4}(iii),
$
v_2(B_n(3,0))
=
1+
v_2 \big( 
a(a+1)
\big)
$.

When $n \equiv 3 \pmod{4}$, it
follows from
Theorem~\ref{Bkj:theorem4}(i)
with $k=3$ and $j=0$
that
$
v_2(B_n(3,0)) = 0$.
\end{proof}

\subsection{k=3, j=1}

\begin{theorem}
\label{Bkj:theoremk3j1}
For any $n \geq 0$,
$$
v_2(B_n(3,1))
=
\begin{cases}
1 + v_2(a), 
& \text{if } 
n=4a
\text{ and }
v_2(a) \leq 4,
\\
0, 
& \text{if } 
n \equiv 1 \pmod{2},
\\
1 + v_2(a), 
& \text{if } 
n=4a+2
\text{ and }
v_2(a) \leq 3
.
  \end{cases}
$$
\end{theorem}
\begin{proof}
When $n=4a$
and $a$ is odd,
it follows from Theorem~\ref{Bkj:theorem5}(i)
with $k=3$, $j=1$ and $r=0$ that
$
v_2(B_n(3,1))
=
1
$.

When $n=4a$
and $a$ is even,
it follows from Theorem~\ref{Bkj:theorem5}(ii)
with $k=3$, $j=1$ and $r=0$ that
if $v_2(a) \leq 3$ then
$
v_2(B_n(3,1))
=
1+
v_2(a)
$.

When $n=4a$ and $v_2(a) =4$,
it follows from
Lemma~\ref{Bkj:lemma5}
with $k=3$ and $j=1$ that
$
v_2(B_n(3,1))
=
5
.$

When $n \equiv 1 \pmod{2}$,
it follows from
Theorem~\ref{Bkj:theorem4}(i)
with $k=3$ and $j=1$ that $
v_2(B_n(3,1)) = 0$.

When $n=4a+2$,
it follows from Theorem~\ref{Bkj:theorem5}(vi)
with $k=3$, $j=1$ and $r=2$ that
if $v_2(a) \leq 3$ then
$
v_2(B_n(3,1))
=
1 + v_2(a)
$.
\end{proof}

\subsection{k=3, j=2}

\begin{theorem}
\label{Bkj:theoremk3j2}
For any $n \geq 0$,
$$
v_2(B_n(3,2))
=
\begin{cases}
1, 
& \text{if } 
n \equiv 4 \pmod{8}
,
\\
3, 
& \text{if } 
n \equiv 8 \pmod{16}
,
\\
1 + v_2 \big( a(a+1) \big), 
& \text{if } 
n=4a+1 \text{ and }
 v_2\big( a(a+1) \big) \leq 4,
 \\
 0, 
& \text{if } 
n \equiv 2 \pmod{4},
\\
0, 
& \text{if } 
n \equiv 3 \pmod{4}.
  \end{cases}
$$
\end{theorem}
\begin{proof}
When $n \equiv 4 \pmod{8}$,
it follows from Theorem~\ref{Bkj:theorem5}(i)
with $k=3$, $j=2$ and $r=0$
that $v_2(B_n(3,2))=1$.

When $n \equiv 8 \pmod{16}$,
it follows from Lemma~\ref{Bkj:lemma6}
with $k=3$ and $j=2$
that $v_2(B_n(3,2))=3$.

When $n = 4a+1$ and $a$ is odd,
it follows from Theorem~\ref{Bkj:theorem5}(iii)
with $k=3$, $j=2$ and $r=1$
that if 
$v_2(a+1) \leq 4$,
then
$
v_2(B_n(3,2))
=
1+
v_2(a+1)
$.

When $n = 4a+1$ and $a$ is even,
it follows from Theorem~\ref{Bkj:theorem5}(iv)
with $k=3$, $j=2$ and $r=1$
that if $v_2(a) \leq 4$,
then
$
v_2(B_n(3,2))
=
1+
v_2(a)
$.

When $n \equiv 2 \pmod{4}$,
it follows from Theorem~\ref{Bkj:theorem4}(i)
with $k=3$, $j=2$ and $r=2$
that $v_2(B_n(3,2)) = 0$.

When $n \equiv 3 \pmod{4}$,
it follows from Theorem~\ref{Bkj:theorem4}(i)
with $k=3$, $j=2$ and $r=-1$
that $v_2(B_n(3,2)) = 0$.
\end{proof}

\section{The $k=4$ case}
\subsection{k=4, j=0}

\begin{theorem}
\label{Bkj:theoremk4j0}
For any $n \geq 0$,
$$
v_2(B_n(4,0))
=
\begin{cases}
0, 
& \text{if } 
n \equiv 0 \pmod{5},
\\
1 + 2 v_2(a), 
& \text{if } 
n=5a + 1
\text{ and } 
v_2(a) \leq 2
\\
1+ v_2\big(a(a+1)\big), 
& \text{if } 
n=5a + 2
\text{ and }
v_2\big(a(a+1)\big) \leq 5,
\\
3, 
& \text{if } 
n=5a + 3
\text{ and }
v_2(a+3) \geq 2,
\\
0, 
& \text{if } 
n \equiv 4 \pmod{5}.
  \end{cases}
$$
\end{theorem}
\begin{proof}
When $n \equiv 0 \pmod{5}$,
it follows from Theorem~\ref{Bkj:theorem4}(i)
with $k=4$, $j=0$ and $r=0$
that $v_2(B_n(4,0))=0$.

When $n=5a + 1$,
$a$ is odd,
it follows from Theorem~\ref{Bkj:theorem5}(v)
with $k=4$, $j=0$ and $r=1$
that $v_2(B_n(4,0))=1$.

When $n=5a + 1$,
$a$ is even and $v_2(a) \leq 2$,
it follows from
Lemma~\ref{Bkj:lemma7}
that 
$
v_2(B_n(4,0))
=
1 + 2 v_2(a)
$.

When $n=5a + 2$,
apply Theorem~\ref{Bkj:theorem4}(iii)
with $k=4$, $j=0$ and $r=2$.
In this case, $\Delta_3= 
7-
v_2 \left( 
{{a+ 1} \choose {a-1}}
2^{2}
\right)
=
6-
v_2\big(a(a+1)\big)
$
and
$
\Delta_4=
v_2 \left( 
{{a+ 2} \choose {a-1}}
2^{3}
\right)
-
v_2 \left( 
{{a+ 1} \choose {a-1}}
2^{2}
\right) 
=
1+
v_2(a+2)
$.
So if 
$v_2\big(a(a+1)\big)  \leq 5$
then
$\Delta_3 >0$ and $\Delta_4 >0$,
and thus
by Theorem~\ref{Bkj:theorem4}(iii),
$
v_2(B_n(4,0))
=
2+
v_2 \left( 
{{a+ 1} \choose {a-1}}
\right)
=
1+
v_2\big(a(a+1)\big) 
$.

When $n=5a + 3$,
apply Theorem~\ref{Bkj:theorem4}(iii)
with $k=4$, $j=0$ and $r=3$.
In this case,
$\Delta_3= 
8-
v_2 \left( 
{{a+ 2} \choose {a-1}}
2^{3}
\right)
=
6-
v_2 \big( 
a(a+1)(a+2)
\big)
$
and
$
\Delta_4=
v_2 \left( 
{{a+ 3} \choose {a-1}}
2^{4}
\right)
-
v_2 \left( 
{{a+ 2} \choose {a-1}}
2^{3}
\right) 
= 
v_2(a+3)  
-1
$.
So if 
$v_2(a+3) \geq 2$
then $\Delta_4 >0$ and
$\Delta_3 =
6-
v_2(a+1)=5>0$,
and thus
by Theorem~\ref{Bkj:theorem4}(iii),
$
v_2(B_n(4,0))
=
3+
v_2 \left( 
{{a+ 2} \choose {a-1}}
\right)
=
2+
v_2 \big(
a(a+1)(a+2) 
\big)
=
2+
v_2(a+1)=3
$.

When $n \equiv 4 \pmod{5}$,
it follows from Theorem~\ref{Bkj:theorem4}(i)
with $k=4$, $j=0$ and $r=-1$
that $v_2(B_n(4,0))=0$.
\end{proof}

\subsection{k=4, j=1}

\begin{theorem}
\label{Bkj:theoremk4j1}
For any $n \geq 0$,
$$
v_2(B_n(4,1))
=
\begin{cases}
1+v_2(a), 
& \text{if } 
n=5a 
\text{ and } 
v_2(a) \leq 5,
\\
0, 
& \text{if } 
n \equiv 1 \pmod{5},
\\
1+ v_2(a), 
& \text{if } 
n=5a + 2
\text{ and }
v_2(a) \leq 4,
\\
1+
v_2 \big( 
a(a+1)
\big), 
& \text{if } 
n=5a + 3
\text{ and }
v_2 \big( 
a(a+1)
\big) \leq 5,
\\
0, 
& \text{if } 
n \equiv 4 \pmod{5}.
  \end{cases}
$$
\end{theorem}
\begin{proof}
When $n=5a$ and $a$ is odd,
it follows from Theorem~\ref{Bkj:theorem5}(i)
with $k=5$, $j=1$ and $r=0$
that $v_2(B_n(4,1))=1$.

When $n=5a$ and $a$ is even,
it follows from Theorem~\ref{Bkj:theorem5}(ii)
with $k=5$, $j=1$ and $r=0$
that if $v_2(a) \leq 4$, then $v_2(B_n(4,1))=1+v_2(a)$.

When $n=5a$ and $v_2(a) =5$,
it follows from 
Lemma~\ref{Bkj:lemma5} that $v_2(B_n(4,1))=6$.

When $n \equiv 1 \pmod{5}$,
it follows from Theorem~\ref{Bkj:theorem4}(i)
that $v_2(B_n(4,1)) = 0$.

When $n = 5a + 2$,
it follows from Theorem~\ref{Bkj:theorem5}(vi)
with $k=4$, $j=1$ and $r=2$ that
if $v_2(a) \leq 4$, then
$
v_2(B_n(4,1))
=
1 + v_2(a)
$.

When $n=5a + 3$,
apply Theorem~\ref{Bkj:theorem4}(iii)
with $k=4$, $j=1$ and $r=3$.
In this case,
$\Delta_3= 7-
v_2 \left( 
{{a+1} \choose {a-1}}
2^{2}
\right)
=6-
v_2 \big( 
a(a+1)
\big)
$
and
$\Delta_4=
v_2 \left( 
{{a+ 3} \choose {a-1}}
2^{4}
\right)
-
v_2 \left( 
{{a+1} \choose {a-1}}
2^{2}
\right) 
=
v_2 \big( 
(a+2)(a+3) 
\big)
$.
So if $v_2 \big( 
a(a+1)
\big) \leq 5$,
then $\Delta_3 >0$ and
$\Delta_4 >0$,
and therefore,
$v_2(B_n(4,1))
=
2+
v_2 \left( 
{{a+ 1} \choose {a-1}}
\right)
=
1+
v_2 \big( 
a(a+1)
\big)
.
$

When $n \equiv 4 \pmod{5}$,
it follows from Theorem~\ref{Bkj:theorem4}(i)
that $v_2(B_n(4,1)) = 0$.
\end{proof}

\subsection{k=4, j=2}

\begin{theorem}
\label{Bkj:theoremk4j2}
For any $n \geq 0$,
$$
v_2(B_n(4,2))
=
\begin{cases}
1+v_2(a), 
& \text{if } 
n=5a
\text{ and } 
v_2(a) \leq 5,
\\
1+ v_2 \big( a(a+1) \big), 
& \text{if } 
n = 5a+1
\text{ and }
v_2 \big( a(a+1) \big) \leq 5,
\\
0, 
& \text{if } 
n \equiv 2 \pmod{5},
\\
1 + v_2(a), 
& \text{if } 
n=5a+3
\text{ and } 
v_2(a) \leq 4,
\\
0, 
& \text{if } 
n \equiv 4 \pmod{5}.
  \end{cases}
$$
\end{theorem}
\begin{proof}
When $n=5a$ and $a$ is odd,
it follows from Theorem~\ref{Bkj:theorem5}(i)
with $k=4$, $j=2$ and $r=0$
that $v_2(B_n(4,2))=1$.

When $n=5a$ and $a$ is even,
it follows from Theorem~\ref{Bkj:theorem5}(ii)
with $k=4$, $j=2$ and $r=0$
that if $v_2(a) \leq 4$, then $v_2(B_n(4,2))=1+v_2(a)$.

When $n=5a$ and $v_2(a) =5$,
it follows from 
Lemma~\ref{Bkj:lemma5} that $v_2(B_n(4,2))=6$.

When $n = 5a+1$
and $a$ is odd,
it follows from Theorem~\ref{Bkj:theorem5}(iii)
with $k=4$, $j=2$ and $r=1$
that if $v_2(a+1) \leq 5$,
then
$v_2(B_n(4,2))=1+v_2(a+1)$.

When $n = 5a+1$
and $a$ is even,
it follows from Theorem~\ref{Bkj:theorem5}(iv)
with $k=4$, $j=2$ and $r=1$
that if $v_2(a) \leq 5$,
then
$v_2(B_n(4,2))=1+v_2(a)$.

When $n \equiv 2 \pmod{5}$,
it follows from Theorem~\ref{Bkj:theorem4}(i)
that $v_2(B_n(4,2)) = 0$.

When $n=5a+3$,
it follows from 
Theorem~\ref{Bkj:theorem5}(iv)
with $k=4$, $j=2$ and $r=3$
that if $v_2(a) \leq 4$, then
$v_2(B_n(4,2))=1 + v_2(a)$.

When $n \equiv 4 \pmod{5}$,
it follows from Theorem~\ref{Bkj:theorem4}(i)
that $v_2(B_n(4,2)) = 0$.
\end{proof}

\subsection{k=4, j=3}

\begin{theorem}
\label{Bkj:theoremk4j3}
For any $n \geq 0$,
$$
v_2(B_n(4,3))
=
\begin{cases}
1 + 2 v_2(a), 
& \text{if } 
n=5a
\text{ and }
v_2(a) \leq 2,
\\
1 + v_2 \big( a(a+1) \big), 
& \text{if } 
n = 5a+1
\text{ and } 
v_2 \big( a(a+1) \big) \leq 5,
\\
3, 
& \text{if } 
n=5a + 2
\text{ and }
v_2(a-1) \geq 2,
\\
0, 
& \text{if } 
n \equiv 3 \pmod{5},
\\
0, 
& \text{if } 
n \equiv 4 \pmod{5}.
  \end{cases}
$$
\end{theorem}
\begin{proof}
When $n=5a$ and $a$ is odd,
it follows from Theorem~\ref{Bkj:theorem5}(i)
with $k=5$, $j=3$ and $r=0$
that $v_2(B_n(4,3))=1$.

When $n=5a$ and $a$ is even,
it follows from Lemma~\ref{Bkj:lemma6}
that if $v_2(a) \leq 2$, then 
$v_2(B_n(4,3))=1 + 2 v_2(a)$.

When $n = 5a+1$
and $a$ is odd,
it follows from Theorem~\ref{Bkj:theorem5}(iii)
with $k=4$, $j=3$ and $r=1$
that if $v_2(a+1) \leq 5$,
then
$v_2(B_n(4,3))=1+v_2(a+1)$.

When $n = 5a+1$
and $a$ is even,
it follows from Theorem~\ref{Bkj:theorem5}(iv)
with $k=4$, $j=3$ and $r=1$
that if $v_2(a) \leq 5$,
then
$v_2(B_n(4,3))=1+v_2(a)$.

When $n=5a + 2$,
apply Theorem~\ref{Bkj:theorem4}(ii)
with $k=4$, $j=3$ and $r=2$.
In this case,
$\Delta_1 = 8 - 
v_2 \left(
{{a+2} \choose {a-1}}
2^{3}
\right)
=
6 - 
v_2 \big(
a(a+1)(a+2)
\big)
$
and
$\Delta_2 = 
v_2 \left(
{{a + 2} \choose {a-2}} 2^{4}
\right)
-
v_2 \left(
{{a+2} \choose {a-1}}
2^{3}
\right)
= 
v_2(a-1) - 1$.
So if $v_2(a-1) \geq 2$,
then $\Delta_2 >0$ and
$\Delta_1 = 6 - v_2(a+1)=5>0$,
and therefore,
$
v_2(B_n(4,3))
=
3+
v_2 \left( 
{{a + 2} \choose {a-1}}
\right)
=
2+
v_2 \big( 
a(a+1)(a+2)
\big)
=
2+
v_2(a+1)
=3
$.

When $n \equiv 3 \pmod{5}$,
it follows from Theorem~\ref{Bkj:theorem4}(i)
that $v_2(B_n(4,3)) = 0$.

When $n \equiv 4 \pmod{5}$,
it follows from Theorem~\ref{Bkj:theorem4}(i)
that $v_2(B_n(4,3)) = 0$.
\end{proof}

\end{document}